\documentclass[a4paper,12pt]{article}
\usepackage[utf8]{inputenc}
\usepackage{amsmath,amsthm,amsfonts,latexsym,amssymb,bm}

\DeclareFontFamily{OT1}{pzc}{}
\DeclareFontShape{OT1}{pzc}{m}{it}%
              {<-> s * pzcmi8t}{}
\DeclareMathAlphabet{\mathpzc}{OT1}{pzc}%
                                {m}{it}

\newcommand{\A}{\mathpzc{a}}
\newcommand{\ax}{\mathpzc{a}}

\newcommand{\B}{\mathpzc{b}}
\newcommand{\bx}{\mathpzc{b}}
\newcommand{\bxx}{\mathpzc{b}(x)}
\newcommand{\D}{\mathpzc{d}}
\newcommand{\dx}{\mathpzc{d}}
\newcommand{\dxx}{\mathpzc{d}(x)}
\newcommand{\h}{\mathpzc{h}}
\newcommand{\hx}{\mathpzc{h}}

\newcommand{\hu}{\hat{u}_\mu}
\newcommand{\hut}{\hat{u}_{\mu_2}}
\newcommand{\hus}{\hat{u}_{\mu_6}}
\newcommand{\but}{\underline{u}_{\mu_2}}
\newcommand{\bu}{\underline{u}_\mu}
\newcommand{\ft}{(\overline{j_m\!\!}\,\,)_{t}}
\newcommand{\uu}{\underline{u}}
\newcommand{\kk}{k\left(\frac{u}{l\D}\right)}
\newcommand{\kkx}{k\left(\frac{u}{l\D(x)}\right)}
\newcommand{\G}{\overline{G}}
\newcommand{\llll}{\hat{l}}

\newcommand{\Cthree}{C_3}
\newcommand{\Cfour}{C_4}
\newcommand{\Cfive}{C_5}
\newcommand{\Csix}{C_6}
\newcommand{\Cseven}{C_7}
\newcommand{\Ceight}{C_9}
\newcommand{\Cnine}{C_8}

\newtheorem{Thm}{Theorem}[section]

\newtheorem{Lem}[Thm]{Lemma}

\newtheorem{Prop}[Thm]{Proposition}

\newtheorem{Rem}[Thm]{Remark}

\newenvironment{altproof}[1]
{\noindent {\bf Proof of {#1}}.} {\nopagebreak\mbox{}\hfill
$\Box$\par\addvspace{0.5cm}}

\newcommand{\R}{\mathbb{R}}

\newcommand{\eps}{\varepsilon}
\newcommand{\weak}{\rightharpoonup}

\begin{document}
\begin{center}
{\Large\bf Positive solutions to logistic type equations with
harvesting}\\
\ \\
Pedro Girão\footnote{Email: pgirao@math.ist.utl.pt. Partially
supported by the Center for Mathematical Analysis, Geometry and
Dynamical Systems through FCT Program POCTI/FEDER.} and Hossein
Tehrani\footnote{Email: tehranih@unlv.nevada.edu. This work was
initiated while the author was visiting IST Lisbon on a sabbatical
from UNLV. The support of both institutions is gratefully
acknowledged.}
\\

\vspace{2.2mm}

{\small 
Center for Mathematical Analysis, Geometry and Dynamical Systems,\\
Instituto Superior Técnico,\\ Av.\ Rovisco Pais,
1049-001 Lisbon, Portugal\\
and\\
Department of Mathematical Sciences,\\
University of Nevada,\\
Las Vegas, NV 89154-4020, USA}
\end{center}

\begin{center}
{\bf Abstract}\\
\end{center}
We use comparison principles, variational arguments and a
truncation method to obtain positive solutions to logistic type
equations with harvesting both in $\R^N$ and in a bounded domain
$\Omega\subset\R^N$, with $N\geq 3$, when the carrying capacity of
the environment is not constant. By relaxing the growth assumption
on the coefficients of the differential equation we derive a new
equation which is easily solved. The solution of this new equation
is then used to produce a positive solution of our original
problem.

\noindent

\section{Introduction}

In this paper we mainly study the existence of {\em positive}\/
solutions to the problem
\begin{equation}\label{eqab}
\left\{
\begin{array}{ll} -\Delta u=\lambda \ax  u-\bx g(u)-\mu \hx & \mbox{ in } \Omega, \\
u=0 & \mbox{ on } \partial\Omega,
\end{array}\right.
\end{equation}
when $\Omega= \R^N$, in which case the boundary condition is
understood as $\lim_{|x|\rightarrow\infty}u(x)=0\,$, as well as
when $\Omega\subset \R^N$ is a bounded smooth domain. Here $N\geq
3$, and both the functions $\ax$, $\bx$, $\hx$, and the parameters
$\lambda$, $\mu$ are nonnegative. Problem~(\ref{eqab}) can be
thought of as the steady state of the reaction-diffusion equation
$$ 
\left\{
\begin{array}{ll}
\frac{\partial u}{\partial t}=\Delta u+\lambda\A
u-\bx g(u)-\mu\hx & x\in \Omega, \\
u(x,0)=u_0(x) & x\in \Omega, \\
u(x,t)=0 & (x,t)\in\partial\Omega\times[0,\infty).
\end{array}\right.
$$ 
We interpret this as the evolution equation arising from the
population biology of
one species. 
As such the function $u$ represents the population density of the
species. Throughout we assume that
\begin{equation}\label{suplin}
\lim_{s\rightarrow 0}\frac{g(s)}{s}=0\qquad{\rm and}\qquad
\lim_{s\rightarrow\infty}\frac{g(s)}{s}=\infty,
\end{equation}
so that the nonlinearity $\lambda \ax  u-\bx g(u)$ represents a
logistic type growth. Furthermore note that both coefficients
$\ax$ and $\bx$ depend on the spatial variable, indicating
variable linear growth and competition rates in the environment.
The function $\h$ is interpreted as the harvesting distribution
and $\mu\h$ is the harvesting rate.  Hence, such equations have
been used, for example, to model fishery or hunting management
problems. We refer to \cite{OSS1} for further historical
background and references. Intuitively, one expects the survival
of the species, i.e.\ the existence of a positive solution to
(\ref{eqab}), only for small values of $\mu$.

Mathematically, the presence of the harvesting term introduces a
number of challenging issues in the study of existence of positive
solutions. Indeed the harvesting term makes the right hand side of
the equation negative at $u=0$, and therefore our problem belongs
to the class of so called semi-positone problems (see \cite{CMS}).
This prevents the direct application of the maximum principle.

The main inspiration for our study was the recent work \cite{CDT}.
There the authors consider problem~(\ref{eqab}) in $\R^N$ with the
positive
and bounded
function
$\ax \in L^{N/2}(\R^N)\cap L^\infty(\R^N)$, the natural setting
for the eigenvalue problem
$$-\Delta u=\lambda \ax u, \quad u\in {\cal D}^{1,2}(\R^N),$$
where ${\cal D}^{1,2}(\R^N)$ is the completion of $C^1_0(\R^N)$
with respect to the norm\linebreak $\left(\int|\nabla u|^2\right)^{1/2}$. In
addition, they assume that $\frac{g(u)}{u}$ is monotone, $g(u)$
behaves like $u^p, \, p>1$, at infinity and most significantly
$\bx=\ax$. These assumptions play a crucial role in the
variational approach presented in~\cite{CDT}, where, using some
delicate integral inequalities, the authors prove, for a certain
range of $\lambda$, the existence of a positive solution bounded
below by $1/|x|^{N-2}$ at infinity, for $\mu$ sufficiently small.
On the other hand, problem~(\ref{eqab}) was also considered by Du
and Ma in~\cite{DM1} and~\cite{DM2} for $g(u)=u^p$ in the absence
of the harvesting term. The existence of a positive solution was
then proved with {\it no restriction} on the growth of the
nonnegative function $\bx$.

Our first motivation for this work was to study the existence of a
positive solution in $\R^N$ in the presence of harvesting under
minimal restriction on the growth of $\bx$. The novelty of our
approach is that it not only enables us to relax the hypotheses on
the nonlinear term $g(u)$ to the more natural
conditions~(\ref{suplin}), so that it does not require the usual
monotonicity and power-like behavior, but also, more importantly,
that it allows for consideration of a broad class of functions
$\bx$. In particular we will be able to handle some functions
$\bx$ satisfying $\bx(x)\to+\infty$ as $|x|\rightarrow\infty$,
reflecting the assumption that the life conditions are less and
less favorable as one moves to infinity.

In our approach we are naturally led to consider equations of the
form
\begin{equation}\label{eq4}
-\Delta u=\lambda \ax
u\left[1-k\left(\frac{u}{\D}\right)\right]-\mu \hx,
\end{equation}
where $k$ is increasing and $\dx$ is a given function. We note
that this reduces to the classical logistic model if $k(u)=u$ and
$\D$ is a constant. Therefore in line with the classical
terminology, letting $\varsigma=\max k^{-1}(1)$, one may call
$\varsigma\D$ the carrying capacity of the environment because
without harvesting or diffusion the growth rate of the population,
$\lambda\A u\left[1-k\left(\frac{u}{\D}\right)\right]$, is
negative for $u>\varsigma\D$.

As it turns out, for suitable choices of the function $\D$
equation (\ref{eq4}) is relatively simple to solve. In fact, using
variational arguments, the maximum principle and comparison
principles, we first prove the existence of a positive solution to
(\ref{eq4}). Afterwards this solution is used to obtain a solution
of the original problem decaying at infinity not faster than $\D$.
Our method is not only simpler than that in \cite{CDT} but also
provides more general results under less restrictive hypotheses on
the coefficients.

In Section~\ref{bounded} we apply the ideas developed to deal with
the case of whole space $\R^N$ to the bounded domain case. This in
particular allows us to consider the situation where $\bx$ blows
up at the boundary of $\Omega$, which to our knowledge has not
been considered before. Indeed since the  boundary of $\Omega$ is
hostile to the population, it is natural to assume that the
carrying capacity of the environment should go to zero at
$\partial\Omega$. The blow up of $\B$ at the boundary of the
domain can then be interpreted as a consequence of the vanishing
of the carrying capacity of the environment at the boundary of the
domain. Our analysis will show that in some sense it is natural to
consider a carrying capacity for the environment that is
proportional to the distance to $\partial\Omega$.
Our results in this chapter complement and extend known results in
the bounded domain case (see \cite{OSS1}).

The organization of the paper is as follows. In Section~\ref{dois}
we state our hypotheses and make some preliminary observations. We
set up problem~(\ref{eqab}) in $\R^N$ when $\B$ does not grow
``too fast." In Section~\ref{especial} we consider
equation~(\ref{eq4}) and obtain a solution for this equation. The
existence of a {\it positive} solution for (\ref{eq4}) is then
proved in Section~\ref{psp}. In Section~\ref{cinco} we use this
solution to get a positive solution to (\ref{eqab}) when the
function $\B$ grows not faster than a certain power of the
distance to the origin. In Section~\ref{fast} we discuss the case
when the function $\B$ does not satisfy the growth requirements of
the previous section. Section~\ref{bounded} deals with the case of
a bounded domain. In Section~\ref{extensions} we generalize to the
case where the function $g$ also depends on the spatial variable.
Finally, in the Appendix we prove some auxiliary results.

Throughout we denote by ${\cal H}:={\cal D}^{1,2}(\R^N)$, $N\geq
3$, and $\|u\|=\|u\|_{{\cal D}^{1,2}(\R^N)}=\left(\int|\nabla
u|^2\right)^{1/2}$ the norm on ${\cal H}$. When the region of
integration is omitted it is understood to be $\R^N$.

\section{The setup in $\R^N$}
\label{dois}

We wish to prove the existence of a positive weak solution to the
equation
\begin{equation}\label{eq}
-\Delta u=\lambda \ax  u-\bx g(u)-\mu \hx ,\qquad u\in {\cal H}.
\end{equation}
We define a weak solution to be a function $u\in {\cal H}$
satisfying
\begin{equation}\label{int}
\int\nabla u\cdot\nabla v=\lambda\int \ax uv-\int \bx
g(u)v-\mu\int \hx v
\end{equation}
for all $v\in{\cal D}(\R^N)$. We state our assumptions.
\begin{enumerate}
\item[({\bf H}$\A$)]The function $\A:\R^N\to\R$ is positive and
belongs to $L^{N/2}(\R^N)\cap L^\infty(\R^N)$.
\end{enumerate}
We call
$$
\lambda_1=\inf_{u\in {\cal H}\setminus\{0\}}\frac{\|u\|^2}{\int\A
u^2}.
$$
\begin{enumerate}
\item[({\bf H}$g$)] The function $g:\R\to\R^+_0$ is continuous,
with $g(s)=0$ for $s\leq 0$. Furthermore, it satisfies
\begin{equation}\label{zero}
\limsup_{s\to 0} \frac{g(s)}{s^{1+\beta}}<\infty,
\end{equation}
where $\beta>0$ is a fixed constant, and
\begin{equation}\label{infinity}
\lim_{s\to+\infty}\frac{g(s)}{s}=+\infty.
\end{equation}
\item[({\bf H}$\B$)] The measurable function $\B :\R^N\to\R$ is
nonnegative, not identically equal to zero, and satisfies
\begin{equation}\label{limsup}
\bx \leq C_1 \ax {\dx}^{-\beta}
\end{equation}
for some $C_1 >0$, where $\D:\R^N\to\R$ is the Aubin-Talenti
instanton defined by
\begin{equation}\label{def_d}
\dxx=(1+|x|^2)^{-(N-2)/2}.
\end{equation}
Let $B_0=\left\{x\in\R^N: \bxx =0\right\}$. We assume either $B_0$
has measure zero, or $B_0=\overline{{\rm int}\,B_0}$ with
$\partial B_0$ Lipschitz.
\end{enumerate}
In the former case we set $\lambda_*=+\infty$ and in the latter
case
$$
\lambda_*=\inf_{u\in {\cal D}^{1,2}({\rm
int}\,B_0)\setminus\{0\}}\frac{\int_{B_0}|\nabla
u|^2}{\int_{B_0}\A u^2}.
$$
By the unique continuation principle (\cite[p.\ 519]{simon})
$\lambda_1<\lambda_*$.
\begin{enumerate}
\item[({\bf H}$\lambda$)] The value $\lambda$ is such that
$\lambda_1<\lambda<\lambda_*$. \item[({\bf H}$\h$)] The
nonnegative and not identically equal to zero function $\h$
belongs to the space $\h\in L^1(\R^N)\cap L^q(\R^N)\cap
L^s(\R^N)$, for some $q>\frac{N}{2}$ and some $s>N$, and there
exists a constant $C_2>0$ such that
\begin{equation}\label{alo}
R^{N/r}|\h|_{L^q(\R^N\setminus B_{R}(0))}\leq C_2\ \ {\rm for\
all}\ R\in\R^{+}
\end{equation}
with $\frac{1}{q}+\frac{1}{r}=1$. Here $B_{R}(0)$ denotes the ball
centered at zero with radius $R$. \item[({\bf H}$\mu$)] The
parameter $\mu$ is nonnegative.
\end{enumerate}
\begin{Rem}
Under the above hypotheses any positive weak solution $u$
of\/~{\rm (\ref{eq})} belongs to $C^{1,\alpha}_{{\rm loc}}(\R^N)$.
Furthermore, $\lim_{|x|\to\infty}u(x)=0$.
\end{Rem}
Indeed, $u$ satisfies
$$-\Delta u-\lambda \ax  u\leq 0.$$
Therefore by \cite[Theorem~8.17]{GT}, for any $x\in\R^N$, we have
$$
\sup_{B_1(x)}u\leq C|u|_{L^{2N/(N-2)}(B_2(x))}\leq C\|u\|\leq C.
$$
So $u\in L^\infty(\R^N)$, and $\lim_{|x|\to\infty}u(x)=0$. From
elliptic regularity theory \cite{GT}, it follows $u\in
C^{1,\alpha}_{{\rm loc}}(\R^N)$. We use the letter $C$ to
represent various positive constants.

The setting in which we make assumption ({\bf H}$\lambda$) is
clarified in
\begin{Prop}\label{exemplo}
Suppose $u\in{\cal H}$ is a positive weak solution to {\rm
(\ref{eq})}.
\begin{enumerate}
\item[{\rm\bf (i)}] The value $\lambda$ satisfies
$\lambda_1\leq\lambda$. This inequality is strict if $\mu>0$ or if
the restriction  of $g$ to $\R^+$ is positive.
\end{enumerate}
Suppose in addition\/ ${\rm int}\, B_0\neq\emptyset$.
\begin{enumerate}
\item[{\rm\bf (ii)}] If $\h=0$ on $B_0$, then $\lambda<\lambda_*$.
\item[{\rm\bf (iii)}] The inequality $\lambda<\lambda_*$ might not
hold if $h\not\equiv 0$ on $B_0$ and $\mu>0$.
\end{enumerate}
\end{Prop}
The proof is given in the Appendix so that we focus first on the
more important part of the paper. In the sequel we will sometimes
abbreviate weak solution to solution.

\section{A related problem}\label{especial}

   From (\ref{zero}) there exists $0<s_0\leq 1$ and $\Cfour>1$ such that
$$
\frac{g(s)}{s}\leq\lambda\frac{\Cfour}{C_1}s^\beta\quad{\rm for}\
s\leq s_0.
$$
We may assume $\Cfour\geq\frac{1}{s_0^\beta}$. We take
\begin{equation}\label{l}
l:=\left(\frac{1}{\Cfour}\right)^{1/\beta},
\end{equation}
so
\begin{equation}\label{lmenors}
l\leq s_0.
\end{equation}
Using (\ref{limsup}),
$$
\B\,\frac{g(s)}{s}\leq\lambda\A\left(\frac{s}{l\D(x)}\right)^\beta\quad{\rm
for}\ s\leq s_0.
$$
We define
\begin{equation}\label{kkkk}
k(s)=s^\beta
\end{equation}
for $s>0$, $k(s)=0$ for $s\leq 0$. We have
\begin{equation}\label{menor}
\B g(s)\leq \lambda\A s k\left(\frac{s}{l\D}\right)\qquad {\rm
for}\ s\leq s_0.
\end{equation}
We first consider the equation
\begin{equation}\label{boundb}
-\Delta u=\lambda \ax  u\left[1-\kk\right]-\mu \hx .
\end{equation}
Although we are primarily interested in the case where $k$ is as
in (\ref{kkkk}), we more generally assume
\begin{enumerate}
\item[({\bf H}$k$)] $k(s)=0$ for $s\leq 0$, $k$ is continuous,
increasing (not necessarily strictly) and $k(\varsigma)=1$ for
some $\varsigma>0$.
\end{enumerate}
In this and the next sections instead of ({\bf H}$\lambda$) we
assume
\begin{enumerate}
\item[({\bf H}$\lambda$)$^\prime$] The value $\lambda$ is such
that $\lambda>\lambda_1$.
\end{enumerate}

\begin{Thm}\label{thm0}
Under\/ {\rm ({\bf H}$\A$)}, {\rm ({\bf H}$k$)}, {\rm ({\bf
H}$\lambda$)$^\prime$} and\/ {\rm ({\bf H}$\h$)}, there exists\/
$\mu_0>0$ such that for all\/ $0\leq\mu\leq\mu_0$ equation\/ {\rm
(\ref{boundb})} has a positive weak solution $\bu\in {\cal H}\cap
C^{1,\alpha}_{{\rm loc}}(\R^N)$. Furthermore, there exists\/
$\Cthree>0$ such that for all\/ $0\leq\mu\leq\mu_0$ this weak
solution\/ $\bu$ satisfies
\begin{equation}\label{baixo}
\bu(x)\geq\frac{\Cthree}{|x|^{N-2}}\quad{\rm for\ large}\ |x|.
\end{equation}
\end{Thm}

In this section we prove existence of a solution to
(\ref{special}) below. This solution will be used in the next
section to establish Theorem~\ref{thm0}. We define $\llll$ by
\begin{equation}\label{llll}
\llll=\varsigma l.
\end{equation}
\begin{Rem}\label{remark}
The function $\llll\D$ is a supersolution of\/ {\rm
(\ref{boundb})}.
\end{Rem}
Indeed, this follows from $-\Delta \D=N(N-2)\D^{2^*-1}>0$, where
$2^*=2N/(N-2)$.
Consider $\G:\R^N\times\R\to\R$ with
$\G(x,u):=\lambda\A(x)\int_0^usk\left(\frac{s}{l\D(x)}\right)\,ds$
and the functional $I_\mu:{\cal H}\to\R\cup\{+\infty\}$ defined by
\begin{equation}\label{I}
I_\mu(u)= \frac{1}{2}\|u\|^2-\frac{\lambda}{2}\int \ax (u^+)^2+
\int \G(\,\cdot\,,u)+\mu\int \hx  u
\end{equation}
if $\int \G(\,\cdot\,,u)<\infty$, and $I_\mu(u)=+\infty$
otherwise. We have used the standard notation $u^+=\max\{0,u\}$.
The function $\D$ belongs to ${\cal H}$. The function $\h$ belongs
to the space $L^{2N/(N+2)}(\R^N)$ because $1<2N/(N+2)<N/2<q$. So
we have $I_\mu(\llll\D)<\infty$ since $\int
\G(\,\cdot\,,\llll\D)<\infty$. Indeed, $k$ increasing in $\R^+$
implies
\begin{equation}\label{boundg}
\G(x,u)\leq \lambda\A(x)u^2\kkx.
\end{equation}
Hence,
\begin{eqnarray*}
\int G(\,\cdot\,,\llll\D)
&\leq&\lambda \llll^2\int \A  \D^2\\
&<& C\|\A \|_{L^{N/2}(\R^N)}\|\D \|^2_{L^{2N/(N-2)}(\R^N)}\\
&\leq&C\left(\int_0^\infty\frac{1}{(1+r^2)^{(N+1)/2}}\,dr\right)^{(N-2)/N}\
\,<\ \,\infty.
\end{eqnarray*}
We define the set
\begin{equation}\label{N}
N=\left\{u\in {\cal H}:u\leq \llll\D \ \,{\rm a.e.\ in}\
\R^N\right\}.
\end{equation}
The set $N$ is weakly closed.

\begin{Lem}\label{coercive} Let $L\geq 0$.
The functional $I_\mu$ is coercive on $N$, uniformly in $\mu$ with
$0\leq\mu\leq L$, i.e.\ for each $C>0$, there exists $R>0$ such
that for all $0\leq\mu\leq L$ and $u\in N$, if $\|u\|>R$ then
$I_\mu(u)>C$.
\end{Lem}
\begin{proof}
Suppose by contradiction there exists $u_n\in N$ with
$\|u_n\|\to\infty$, and $\mu_n\in[0,L]$ such that
$I_{\mu_n}(u_n)\leq C$. The sequence $v_n:=u_n/\|u_n\|$ is bounded
in ${\cal H}$ and so we may assume $v_n\weak v$ in ${\cal H}$,
$v_n\to v$ a.e.\ in $\R^N$. Since $u_n\leq \llll\D $ we have
$v^+\equiv 0$. Thus $\int \ax (v_n^+)^2=o(1)$. Clearly,
$$
I_{\mu_n}(u_n)\geq\|u_n\|^2\left(\frac{1}{2}+o(1)
-C\frac{|\h|_{L^{2N/(N+2)}(\R^N)}}{\|u_n\|}\right)\longrightarrow
\infty.
$$
This contradiction proves the lemma.
\end{proof}
Since the functional $I_\mu$ is weakly lower semi-continuous on
${\cal H}$, it admits a minimizer $\hu$ on $N$ for each $\mu\geq
0$. We note the derivative $I'_\mu(\hu)\varphi$ is well defined
for any $\varphi\in{\cal H}\cap L^\infty(\R^N)$ with compact
support because $\sup \hu$ is uniformly bounded (by $\llll\D$). In
Lemma~\ref{differentiability} we prove the differentiability of a
related functional in a more general situation when we do not know
a priori $\sup \hu$ is uniformly bounded.
\begin{Lem}\label{m}
The function $\hu$ is a solution to the equation
\begin{equation}\label{special}
-\Delta u=\lambda\A u^+-\lambda\A
uk\left(\frac{u}{l\D}\right)-\mu\h.
\end{equation}
\end{Lem}
The argument of the proof is identical to the one in
\cite[subsection I.2.3]{S}.

\begin{Lem}\label{negative}
There exist $\mu_1, \Cfive>0$ such that for $0\leq\mu\leq\mu_1$,
we have $\inf_{N} I_\mu\leq -\Cfive<0$.
\end{Lem}
\begin{proof} From the definition of $\lambda_1$, there exists a sequence
$u_n\in{\cal D}(\R^N)\setminus\{0\}$ satisfying
$$
\frac{\|u_n\|^2}{\int\A u_n^2}\to\lambda_1.
$$
Since
$$
\min\left\{\frac{\|u_n^+\|^2}{\int\A
(u_n^+)^2},\frac{\|u_n^-\|^2}{\int\A (u_n^-)^2}\right\}\leq
\frac{\|u_n\|^2}{\int\A u_n^2}
$$
if $u_n$ changes sign, we may assume each function $u_n$ is
nonnegative. Fix an $n$ large enough so
$$
\frac{\|u_n\|^2}{\int\A u_n^2}<\lambda
$$
and let $K$ be the support of $u_n$. For small $t\in\R^+$, the
energy of $tu_n$ is
\begin{eqnarray*}
I_\mu(tu_n)&=&{\textstyle\frac{t^2}{2}}\|u_n\|^2
-{\textstyle\frac{\lambda t^2}{2}}\int_{K} \ax u_n^2
+\int_{K}  G(\,\cdot\,,tu_n)+\mu t\int_{K} \hx u_n\\
&\leq&{\textstyle\frac{t^2}{2}\|u_n\|^2
\left(1-\lambda\frac{\int_K\A u_n^2}{\|u_n\|^2}\right)} +t^{2}o(1)
+\mu t\int_{K} \hx u_n.
\end{eqnarray*}
Here $o(1)\to 0$ as $t\to 0$. We have used (\ref{boundg}), $k$ is
continuous at zero with $k(0)=0$ and $u_n\in{\cal D}(\R^N)$. Note
$\D^{-1}\in L^\infty(K)$. We fix $t$ small enough so $tu_n\in N$
and the sum of the first two terms is negative, say equal to $-C$,
with $C>0$. For $\mu$ sufficiently small, $0\leq\mu\leq\mu_1$, the
last term can be made smaller than $-C/2$. This shows $\inf_N
I_\mu\leq-C/2=:-\Cfive$.
\end{proof}

As in \cite[Proposition~1.4]{CDT}, there exist $0<r_0<R_0$ such
that
\begin{equation}\label{between}
0\leq\mu\leq\mu_1\ \  \Rightarrow\ \ r_0\leq\|\hu\|\leq R_0.
\end{equation}
Indeed, the inequality
$$
I_\mu(u)\geq -C\|u\|^2+\int G(\,\cdot\,,u)-C\|u\|\geq
-C\|u\|^2-C\|u\|
$$
implies
$$
\liminf_{u\to 0}I_\mu(u)\geq 0.
$$
Thus (\ref{between}) follows from Lemmas~\ref{coercive}
and~\ref{negative}.

\section{A positive solution for the related problem}\label{psp}
In this section we use the minimizers $\hu$ of $I_\mu$ on $N$
obtained above, Lemmas~\ref{coercive} and \ref{negative}, and (\ref{between}) to complete the\\
\begin{altproof}{Theorem~\ref{thm0}}
By the Riesz Representation Theorem there exists $w\in{\cal H}$
satisfying
\begin{equation}\label{h}
\int \nabla w\cdot\nabla \phi=\int \hx\phi
\end{equation}
for all $\phi\in{\cal H}$, as $\h\in L^{2N/(N+2)}$. Since also
$\h\in L^s$ for some $s>N$, by elliptic regularity theory $w$
belongs to the space $C^{1,\alpha}_{{\rm loc}}(\R^N)$ for some
$\alpha>0$. We can rewrite (\ref{special}) as
$$
-\Delta(\hu+\mu w)=\lambda \A  \hu^+\left[1-k\left(\frac{\hu}{l\D
}\right)\right].
$$
The right hand side satisfies $0\leq \lambda \A
\hu^+\left[1-k\left(\frac{\hu}{l\D }\right)\right]\leq \lambda \A
\hu^+$, since $\hu\leq \llll\D $ and $k$ is increasing in $\R^+$.
As $\hu^+\in L^\infty(\R^N)$ and $\A \in L^\infty(\R^N)$, by
elliptic regularity theory $\hu\in C^{1,\alpha}_{{\rm
loc}}(\R^N)$.

There exists $0<\mu_2\leq\mu_1$ such that for all
$0\leq\mu\leq\mu_2$ one can choose $x_0(\mu)$ where
$\hu(x_0(\mu))>0$. Otherwise $\hu\leq 0$ and
\begin{eqnarray*}
I_\mu(\hu)&=&\frac{1}{2}\|\hu\|^2+\mu\int \hx \hu\\
&\geq&\frac{1}{2}\|\hu\|^2-\mu|\h|_{L^{2N/(N+2)}}C\|\hu\|\ \,
\geq\ \, 0
\end{eqnarray*}
for small $\mu$ because $r_0\leq\|\hu\|\leq R_0$ (see
(\ref{between})). This contradicts Lemma~\ref{negative}.

Because the function $\hut$ is a solution of (\ref{special}) for
$\mu=\mu_2$, the function $\hut$ is a subsolution of
(\ref{special}) for $0\leq\mu\leq\mu_2$. Using
Lemma~\ref{coercive}, we minimize the functional $I_\mu$ over the
set
\begin{equation}\label{numero}
M=\left\{u\in{\cal H}:\hut\leq u\leq \llll\D \ {\rm \, a.e.\ in}\
\R^N\right\}.
\end{equation}
Thus, for $0\leq\mu\leq\mu_2$, obtain new solutions $\bu$ of
(\ref{special}), which means
\begin{equation}\label{smooth}
\int\nabla \bu\cdot\nabla v=\lambda\int \ax \bu^+ v-\lambda\int
\A\bu k\left(\frac{\bu}{l\D}\right)v-\mu\int \hx v,
\end{equation}
for all $v\in{\cal D}(\R^N)$.

For later reference, we note that using Lemma~\ref{negative},
inequality (\ref{between}) and observing that
$$I_\mu(\bu)\leq I_{\mu_2}(\hut)+C|\mu - \mu_2|R_0,$$
we may assume, by decreasing $\mu_2$ if necessary, that
\begin{equation}\label{negativity}
I_\mu(\bu)=\inf_{M} I_\mu\leq -\frac{\Cfive}{2}<0,\quad\quad 0\leq
\mu \leq \mu_2.
\end{equation}
Here the constant $C_5$ is as in Lemma~\ref{negative}.

 We fix $x_0=x_0(\mu_2)$. There exists $\rho>0$ such that
$$
\inf_{\overline{B_\rho(x_0)}}\hut>0.
$$
Choose $\eps$ sufficiently small satisfying
$$
\frac{\eps}{|x-x_0|^{N-2}}<\hut(x)=\but(x)\quad{\rm if}\
x\in\partial B_\rho(x_0).
$$
All the $\bu$ lie above $\but$ and $w$ is positive so
\begin{equation}\label{positive}
\inf_{\overline{B_\rho(x_0)\!\!}}\,\bu\geq\inf_{\overline{B_\rho(x_0)\!\!}}\,\but>0
\end{equation}
and
$$
\frac{\eps}{|x-x_0|^{N-2}}<(\bu+\mu w)(x)\quad{\rm if}\
x\in\partial B_\rho(x_0),
$$
for all $0\leq\mu\leq\mu_2$. Let
$$
S_\mu=\left\{x\in
B_\rho(x_0)^C:\frac{\eps}{|x-x_0|^{N-2}}>(\bu+\mu w)(x)\right\}.
$$
Note $ 0\leq \lambda \A  \bu k\left(\frac{\bu}{l\D}\right)\leq
\lambda \A  \bu^+. $ Let $v$ be an arbitrary function in ${\cal
H}$ and $v_n\in{\cal D}(\R^N)$, $v_n\to v$ in ${\cal H}$. Using
equality (\ref{smooth}) with $v$ replaced by $v_n$ and passing to
the limit, we see (\ref{smooth}) is valid for $v$ in ${\cal H}$.
Hence, using (\ref{h}),
\begin{equation}\label{fmu}
\int\nabla(\bu+\mu w)\cdot\nabla\phi=\int \lambda \A
\hu^+\left[1-k\left(\frac{\hu}{l\D}\right)\right] \phi\quad{\rm
for\ all}\ \phi\in {\cal H},
\end{equation}
Also
\begin{equation}\label{fund}
\int\nabla\left(\frac{1}{|x-x_0|^{N-2}}\right)\cdot\nabla\phi=0
\end{equation}
for all $\phi\in {\cal H}$ satisfying $\phi(x)=0$ for $x\in
B_\rho(x_0)$. Subtracting (\ref{fund}) from (\ref{fmu}),
$$
\int\nabla\left(\bu+\mu
w-\frac{\eps}{|x-x_0|^{N-2}}\right)\cdot\nabla\phi= \int \lambda
\A  \hu^+\left[1-k\left(\frac{\hu}{l\D}\right)\right] \phi
$$
for all $\phi\in {\cal H}$ satisfying $\phi(x)=0\ {\rm for}\ x\in
B_\rho(x_0)$. The function $\phi:= \left(\bu+\mu
w-\frac{\eps}{|x-x_0|^{N-2}}\right) \chi_{S_\mu}$ belongs to
${\cal H}$, is less than or equal to zero and has support in
$B_\rho(x_0)^C$. Thus
$$
\int_{S_\mu}\left|\nabla\left(\bu+\mu
w-\frac{\eps}{|x-x_0|^{N-2}}\right)\right|^2\leq 0.
$$
Therefore $S_\mu$ is empty which means
\begin{equation}\label{below}
\frac{\eps}{|x-x_0|^{N-2}}\leq (\bu+\mu w)(x)\quad{\rm for\ all}\
x\in B_\rho(x_0)^C.
\end{equation}
We now recall the following lemma due to Allegretto and Odiobala.
\begin{Lem}{\rm \cite[Lemma 4]{AO}}\label{AO}
Let $\h\in L^1(\R^N)$ and suppose\/ {\rm (\ref{alo})} holds. Then
there exists a constant $C$ such that
$$
w(x)\leq \frac{C}{|x|^{N-2}}\quad{\rm for\ all}\
x\in\R^N\setminus\{0\}.
$$
\end{Lem}
Combining the estimates (\ref{positive}) and (\ref{below}) with
Lemma~\ref{AO}, we conclude there exists $0<\mu_0\leq \mu_2$ such
that for all $0\leq\mu\leq\mu_0$ the function $\bu$ is positive
and $\bu(x)\geq\frac{\Cthree}{|x|^{N-2}}$ for $x\in
B_\rho(x_0)^C$. This completes the proof of Theorem~\ref{thm0}.
\end{altproof}

\section{A positive solution in $\R^N$}\label{cinco}
We now turn to equation (\ref{eq}).
\begin{Thm}\label{thm}
Under\/ {\rm ({\bf H}$\A$)}, {\rm ({\bf H}$g$)}, {\rm ({\bf
H}$\B$)}, {\rm ({\bf H}$\lambda$)} and\/ {\rm ({\bf H}$\h$)},
there exists\/ $\mu_0>0$ such that for all\/ $0\leq\mu\leq\mu_0$
equation\/ {\rm (\ref{eq})} has a positive weak solution $u_\mu\in
{\cal H}\cap C^{1,\alpha}_{{\rm loc}}(\R^N)$. Furthermore, there
exists\/ $\Cthree>0$ such that for all\/ $0\leq\mu\leq\mu_0$ this
weak solution\/ $u_\mu$ satisfies
\begin{equation}\label{large}
u_\mu(x)\geq\frac{\Cthree}{|x|^{N-2}}\quad{\rm for\ large}\ |x|.
\end{equation}
\end{Thm}
\begin{proof}
We take the function $k$ as in (\ref{kkkk}) and apply
Theorem~\ref{thm0} to obtain a positive solution $\bu$ of
(\ref{boundb}) for $0\leq\mu\leq\mu_0$. Using (\ref{menor}) and
\begin{equation}\label{s0}
\bu\leq\llll \D=\varsigma l\D=l\D\leq l\leq s_0
\end{equation}
(see (\ref{numero}), (\ref{llll}), ({\bf H}$k$) and
(\ref{lmenors})), the function $\bu$ satisfies
$$
-\Delta\bu\leq\lambda \ax \bu-\bx g(\bu)-\mu \hx,
$$
and so is a subsolution of our problem.

Fix any $1<p\leq (N+2)/(N-2)$. For all integers $m$ with $m\geq 1$
we define $j_m:\R\to\R$ by
\begin{equation}\label{jj}
j_m(s)=\left\{\begin{array}{ll}
g(s)&{\rm for}\ s\leq m,\\
g(m)-m^p+s^p&{\rm for}\ s>m.
\end{array}\right.
\end{equation}
We also define $j:\R\to\R$ by
$$
j(s)=\inf_{m\geq 1}j_m(s).
$$
The function $j$ is measurable and in $L^1_{{\rm loc}}(\R)$.
\begin{Lem}
The function $j$ satisfies
\begin{equation}\label{j}
\lim_{s\to+\infty}\frac{j(s)}{s}=+\infty.
\end{equation}
\end{Lem}
\begin{proof}
By contradiction, suppose there exists a constant $C>0$ and a
sequence $s_n\to+\infty$ such that $\frac{j(s_n)}{s_n}\leq C$.
Then there also exists a sequence $(m_n)$ with $m_n\geq 1$ and
$$\frac{j_{m_n}(s_n)}{s_n}\leq C+1.$$ From the definition of $j_{m_n}$ and using
$\frac{g(s_n)}{s_n}\to+\infty$, it follows $s_n>m_n$ for large
$n$. So for large $n$
$$
\frac{j_{m_n}(s_n)}{s_n}=\frac{g(m_n)-m_n^p+s_n^p}{s_n}=\frac{g(m_n)-m_n^p}{s_n}+s_n^{p-1}\leq
C+1.
$$
The last inequality implies $g(m_n)<m_n^p$ for large $n$ and
$m_n\to+\infty$. Thus
$$
C+1\geq\frac{j_{m_n}(s_n)}{s_n}\geq\frac{g(m_n)-m_n^p}{m_n}+s_n^{p-1}=
\frac{g(m_n)}{m_n}-m_n^{p-1}+s_n^{p-1}\geq\frac{g(m_n)}{m_n}
$$
for large $n$. From assumption (\ref{infinity}),
$\lim_{n\to\infty}\frac{g(m_n)}{m_n}=+\infty$. We have reached a
contradiction. This proves (\ref{j}).
\end{proof}
For $0\leq\mu\leq\mu_0$ the functions $\bu$ satisfies $0<\bu\leq
\llll\D\leq\llll= l\leq 1\leq m$ (see (\ref{l}) and (\ref{llll})).
Since every $j_m$ coincides with $g$ up to $m$, we have $\bu$
satisfies
$$
-\Delta\bu\leq\lambda \ax \bu-\bx j_m(\bu)-\mu \hx.
$$
For each $0\leq\mu\leq\mu_0$, we define the set
$$
M_\mu=\left\{u\in {\cal H}:\bu\leq u\/\ {\rm\, a.e.\ in}\
\R^N\right\}.
$$
The set $M_\mu$ is weakly closed. Let $J_m(s)=\int_0^sj_m(t)\,dt$
and $J(s)=\int_0^sj(t)\,dt$. The function $J$ is continuous. For
$m\geq 1$ we also define $I^m_\mu:M_\mu\to\R\cup\{+\infty\}$ by
$$
I^m_\mu(u)= \frac{1}{2}\|u\|^2-\frac{\lambda}{2}\int \ax u^2+ \int
\bx J_m(u)+\mu\int \hx  u
$$
if $\int \bx J_m(u)<\infty$, and $I^m_\mu(u)=+\infty$ otherwise.
Similarly, we define $I^0_\mu$ with $J$ in the place of $J_m$.

\begin{Lem}\label{coercivity}
The functionals $I^m_\mu$ are coercive on $M_\mu$, uniformly in
$m$ and $\mu$ with $m\geq 1$ and $0\leq\mu\leq\mu_0$, i.e.\ for
each $L>0$, there exists $R>0$ such that for all $m\geq 1$,
$0\leq\mu\leq\mu_0$ and $u\in M_\mu$, if\/ $\|u\|>R$ then
$I^m_\mu(u)>L$.
\end{Lem}
\begin{proof} The argument is similar to the one in \cite[proof of Theorem~6]{DM2}.
Suppose by contradiction there exists $\mu_n\in[0,\mu_0]$,
$m_n\geq 1$ and $u_n\in M_{\mu_n}$ with $\|u_n\|\to\infty$, such
that $I^{m_n}_{\mu_n}(u_n)\leq C$. From the definition of $j$ we
also have $I^0_{\mu_n}(u_n)\leq C$. Clearly
$$
c_n^2:=\int \ax u_n^2\longrightarrow+\infty
$$
since $J$ is nonnegative, and $\int hu\geq 0$ for all $u\in
M_\mu$. We define a sequence of functions, $(v_n)$, with
$v_n=\frac{u_n}{c_n}$, so that $\int \ax v_n^2=1$ and
\begin{equation}\label{vn}
\frac{1}{2}\|v_n\|^2-\frac{\lambda}{2}+ \frac{1}{c_n^{2}}\int \bx
J(c_nv_n)+\frac{\mu_n}{c_n}\int \hx  v_n\leq \frac{C}{c_n^2}.
\end{equation}
Inequality (\ref{vn}) implies $\|v_n\|$ is uniformly bounded in
$n$. Up to a subsequence, $v_n\weak v$ in ${\cal H}$ and $v_n\to
v$ a.e.\ in $\R^N$. The function $v$ is nonnegative. Inequality
(\ref{j}) implies $\lim_{s\to+\infty}J(s)/s^2=+\infty$. Taking the
limit inferior on both sides of (\ref{vn}), and using Fatou's
lemma,
$$
\frac{1}{2}\|v\|^2-\frac{\lambda}{2}+\int_{\left\{x\in\R^N:v(x)>0\right\}}\bx
\times(+\infty)v^2\leq 0
$$
The function $v$ must be zero almost everywhere on the set where
the function $\B$ is positive, i.e.\ (aside from a set of measure
zero) $v$ must have support in~$B_0$. We also obtain
$\|v\|^2\leq\lambda$. On the other hand, since $\int \A v_n^2=1$
and $\int \A v_n^2\to\int \A v^2$, the function $v\not\equiv 0$
and $\int \A v^2=1$. If $B_0$ has measure zero, then we are done.
Otherwise, ({\bf H}$\B$) implies $v\in{\cal D}^{1,2}({\rm
int}\,B_0)$ and
$$
\lambda_*\leq\frac{\|v\|^2}{\int \A v^2}\leq\lambda.
$$
This contradicts $\lambda<\lambda_*$. The lemma is proved.
\end{proof}

For $0\leq\mu\leq\mu_0$ and $m\geq 1$, the functional $I^m_\mu$
has a minimizer $u^m_\mu$ on $M_\mu$, which of course is positive.
\begin{Lem}\label{differentiability}
Suppose $v\in {\cal H}(\R^N)$ with compact support. For\/
$u\in{\cal H}$ with $\int \bx J_m(u)<\infty$, the functional
$I^m_\mu$ is differentiable in the direction $v$ and $$\textstyle
\left.\frac{d}{dt}\int \bx J_m(u+tv)\right|_{t=0}= \int \bx
j_m(u)v.$$
\end{Lem}
\begin{proof}
Our assumption on $p$ and $\B\in L^\infty_{{\rm loc}}(\R^N)$ imply
$\int \bx J_m(u+tv)<\infty$. Suppose $0<|t|\leq 1$.
\begin{eqnarray*}
\frac{\int \B
[J_m(u+tv)-J_m(u)]}{t}&=&\int_{\left\{x\in\R^N:v(x)\neq 0\right\}}
\B
\left(\frac{1}{tv}\int_u^{u+tv}j_m(s)\,ds\right)v\,dx\\
&=&\int_{\left\{x\in\R^N:v(x)\neq 0\right\}} \B \ft v\,dx,
\end{eqnarray*}
where $\ft :\left\{x\in\R^N:v(x)\neq 0\right\}\to\R$ is defined by
$$\ft (x):=\frac{1}{tv(x)}\int_{u(x)}^{u(x)+tv(x)}j_m(s)\,ds.$$
We have
$$
|\ft |\leq\eps(u^++v^+)+C_\eps((u^+)^{p}+(v^+)^{p}).
$$
The function $\B[\eps(u^++v^+)+C_\eps((u^+)^{p}+(v^+)^{p})]v$ is
integrable. So the assertion of the lemma follows from Lebesgue's
Dominated Convergence Theorem.
\end{proof}
Using Lemma~\ref{differentiability}, $I^m_\mu$ is differentiable
at $u^m_\mu$ in the direction of functions $\varphi$ of compact
support. As in Lemma~\ref{m} one can prove $u^m_\mu$ is a solution
of
\begin{equation}\label{mmm}
-\Delta u=\lambda \ax  u-\bx j_m(u)-\mu \hx,
\end{equation}
by showing $(I^m_\mu)'(u^m_\mu)\varphi=0$ for all $\varphi\in{\cal
D}(\R^N)$. The functions $u^m_\mu$ satisfy
$$-\Delta u^m_\mu-\lambda \ax  u^m_\mu\leq 0.$$
By \cite[Theorem~8.17]{GT} we have
\begin{equation}\label{unif}
\sup_{\R^N}u^m_\mu\leq \Csix\|u^m_\mu\|,
\end{equation}
where the constant $\Csix$ depends only on $N$, $\lambda$ and the
norm $|\A|_{L^\infty(\R^N)}$. Furthermore, from (\ref{menor}),
(\ref{s0}), $s_0\leq 1\leq m$ and (\ref{negativity}), we have
$$I^m_\mu(u^m_\mu)\leq I_\mu(\bu) <0.$$ So using
Lemma~\ref{coercivity} there exists an $R>0$ such that
$\|u^m_\mu\|\leq R$. It follows $\sup_{\R^N}u^m_\mu\leq \Csix
R=:\Cseven$. If we take any constant $m\geq \Cseven$, the function
$u^m_\mu$ is a solution of (\ref{eq}). Since the right-hand-side
of (\ref{eq}) belongs to $L^s_{{\rm loc}}(\R^N)$ and $s>N$ by
elliptic regularity theory $u\in C^{1,\alpha}_{{\rm loc}}(\R^N)$
for some $\alpha>0$. Estimate (\ref{large}) is immediate from
(\ref{baixo}). The proof of Theorem~\ref{thm} is complete.
\end{proof}

Suppose $\tilde{\D}$ is another function satisfying the properties
that we used concerning the function $\D$, i.e.\ suppose
$\tilde{\D}\in {\cal H}$ is continuous, $\tilde{\D}\not\equiv 0$
and $-\Delta\tilde{\D}\geq 0$. Multiplying the last inequality by
$\tilde{\D}^-$ and integrating, $\tilde{\D}^-\equiv 0$. From {\rm
\cite[Theorem~8.19]{GT}}, there exists $C$ such that
$\inf_{x\in\overline{B_{1}(0)}}\tilde{\D}(x)=C>0$. Hence,
\begin{equation}\label{cccc}
\tilde{\D}(x)\geq \frac{C}{|x|^{N-2}}
\end{equation}
for $x\in{\partial B_{1}(0)}$. As $x\mapsto\frac{C}{|x|^{N-2}}$ is
harmonic in $B_1(0)^C$, by the maximum principle inequality {\rm
(\ref{cccc})} also holds for $x\in B_1(0)^C$. So $\tilde{\D}\geq
C\D$. If $\B\leq\tilde{C}_1\A \tilde{\D}^{-\beta}$ for some
constant $\tilde{C}_1>0$, then $\B\leq C_1\A \D^{-\beta}$ for some
constant $C_1>0$. So we cannot apply the proof above if $\B$ grows
faster than in {\rm (\ref{limsup})}. In addition, inequality {\rm
(\ref{large})}  shows the bound $\bu\leq \llll\D$ is sharp.

\section{The case where $\B$ grows fast}\label{fast}

Equation (\ref{eq}) may have positive solutions for $\B$ growing
faster than in (\ref{limsup}), or in other words for $\D$ going
faster to zero than $1/|x|^{N-2}$ as $|x|\to\infty$. We now prove
a theorem regarding such a situation. We will relax the growth
condition on $\B$ at infinity and the condition on $g$ at zero, at
the expense of assuming a more restrictive hypothesis for $\h$.

Instead of ({\bf H}$g$), ({\bf H}$\B$) and ({\bf H}$\h$) we now
assume
\begin{enumerate}
\item[({\bf H}$g$)$^\prime$] The function $g:\R\to\R^+_0$ is
continuous, with $g(s)=0$ for $s\leq 0$. Furthermore,
$$
\lim_{s\to 0}\frac{g(s)}{s}=0
$$
and (\ref{infinity}) holds. \item[({\bf H}$\B$)$^\prime$] The
measurable function $\B :\R^N\to\R$ is nonnegative, not
identically equal to zero, and satisfies $\B=\lambda\A
\varUpsilon$ with $\varUpsilon\in L^\infty_{{\rm loc}}(\R^N)$. Let
$B_0=\left\{x\in\R^N: \varUpsilon(x) =0\right\}$. We assume either
$B_0$ has measure zero, or $B_0=\overline{{\rm int}\,B_0}\neq\R^N$
with ${\rm int}\,B_0\neq\emptyset$ and $\partial B_0$ Lipschitz.
\item[({\bf H}$\h$)$^\prime$] The measurable, nonnegative and not
identically equal to zero function $\h$ has compact support and
there exists a constant $\Cnine$ such that $\h\leq \Cnine\A$.
\end{enumerate}

\begin{Thm}\label{thm2}
Under\/ {\rm ({\bf H}$\A$)}, {\rm ({\bf H}$g$)$^\prime$}, {\rm
({\bf H}$\B$)}$^\prime$, {\rm ({\bf H}$\lambda$)} and\/ {\rm ({\bf
H}$\h$)$^\prime$}, there exists\/ $\mu_3>0$ such that for all\/
$0\leq\mu\leq\mu_3$ equation\/ {\rm (\ref{eq})} has a positive
weak solution \/ $u_\mu\in {\cal H}\cap C^{1,\alpha}_{{\rm
loc}}(\R^N)$. Furthermore, there exists a constant $C>0$ such
that, for all\/ $0\leq\mu\leq\mu_3$,
$\|u_\mu\|_{L^\infty(\R^N)}\leq C$.
\end{Thm}
\begin{proof}
To solve equation (\ref{eq}), we first consider
\begin{equation}\label{comdois}
-\Delta u=\lambda\A u-2\tilde{\B} \tilde{g}(u),
\end{equation}
where $\tilde{\B}=\lambda\A\tilde{\varUpsilon}$, with
$\tilde{\varUpsilon} =\max\{\varUpsilon,1\}$, and
$\tilde{g}(u)=g(u)+(u^+)^2$.
Obviously, zero is a solution to this equation. We define the set
\begin{equation}\label{Mbig}
M=\left\{u\in {\cal H}:u\geq 0\/\ {\rm\, a.e.\ in}\ \R^N\right\}.
\end{equation}
For all integers $m\geq 1$, we define $I^m:M\to\R\cup\{+\infty\}$
by
$$
I^m(u)= \frac{1}{2}\|u\|^2-\frac{\lambda}{2}\int \ax u^2+ 2\int
\tilde{\bx} J_m(u)
$$
if $\int \tilde{\bx} J_m(u)<\infty$, and $I^m(u)=+\infty$
otherwise. Here $J_m$ is as in Section~\ref{cinco} with $g$
replaced by $\tilde{g}$. As in Lemma~\ref{coercivity}, the
functionals $I^m$ are coercive on $M$, uniformly in $m$. Indeed,
$\{x\in\R^N:\tilde{\B}(x)=0\}=\emptyset$.
For $m\geq 1$, the functional $I^m$ has a minimizer $\uu^m$ on
$M$. As a consequence of the analogue of Lemma~\ref{negative},
$I^m(\uu^m)<0$. Lemma~\ref{differentiability} applies as well as
the subsequent discussion. Equation~(\ref{comdois}) has a
nonnegative solution $\uu\in C^{1,\alpha}_{{\rm loc}}(\R^N)$. We
observe that $\uu\not\equiv 0$ since it has negative energy. We
prove that $\uu$ is positive.
We may rewrite (\ref{comdois}) as
$$
-\Delta u=\lambda\A u(1-2\tilde{\varUpsilon} k(u)),
$$
with $k(s)=\tilde{g}(s)/s$ for $s\neq 0$ and $k(0)=0$. Suppose by
contradiction $\uu$ vanishes at some point $x_0$. Because $\uu$
and $k$ are continuous, $k(\uu(x_0))=0$ and
$\tilde{\varUpsilon}\in L^\infty_{{\rm loc}}(\R^N)$, there exist
$r>0$ such that $1-2\tilde{\varUpsilon}(x) k(\uu(x))>0$ for $x\in
B_r(x_0)$. Thus $-\Delta \uu(x)\geq 0$ in the sense of
distributions for $x\in B_r(x_0)$. From \cite[Theorem~8.19]{GT}, it follows $\uu\equiv 0$ in
$B_r(x_0)$. By the unique continuation principle (\cite[p.\
519]{simon}) $\uu\equiv 0$ in $\R^N$. We have reached a
contradiction so $\uu$ is positive.

There exists a constant $c>0$ such that $\uu(x)\geq c$ for $x$ in
the support of $\h$. Then $\tilde{g}(\uu(x))\geq c^2$ for $x$ in
the support of $\h$. Let $0\leq\mu\leq\mu_3:=\frac{\lambda
c^2}{\Cnine}$. Taking into account ({\bf H}$\B$)$^\prime$ and
({\bf H}$\h$)$^\prime$, $\tilde{\B}\geq\lambda\A$ and $\h\leq
\Cnine\A\leq\frac{\Cnine}{\lambda}\tilde{\B}$. Then in the support
of $\h$,
we have
$$
\mu\h\leq\frac{\lambda c^2}{\Cnine}\h\leq c^2\tilde{\B}\leq
\tilde{\B} \tilde{g}(\uu);
$$
thus $\mu\h\leq\tilde{\B} \tilde{g(}\uu)$ everywhere on $\R^N$. So
$\uu$ satisfies
$$
-\Delta \uu\ \leq\ \lambda\A \uu-\tilde{\B} \tilde{g}(\uu)-\mu\h\
\leq\ \lambda\A \uu-{\B}{g}(\uu)-\mu\h.
$$
We also have,
\begin{eqnarray*}
\tilde{I}_\mu(\uu)&:=& \frac{1}{2}\int|\nabla
\uu|^2-\frac{\lambda}{2}\int\A \uu^2+
\int{\B}{G}(\uu)+\mu\int\h \uu\\ 
& \leq & \frac{1}{2}\int|\nabla \uu|^2-\frac{\lambda}{2}\int\A
\uu^2+ \int\tilde{\B}\tilde{G}(\uu)+\mu\int\h \uu\ \leq\ C\ <\
\infty
\end{eqnarray*}
because $I^m(\uu)<0$, and $\h$ has compact support and belongs to
the space $L^\infty(\R^N)$.
(We could even take $C$ to be zero if we restricted
$0\leq\mu\leq\frac{\lambda c^2}{3\Cnine}$ because this would imply
$\mu\int\h \uu\leq\int\tilde{\B}\tilde{G}(\uu)$). Repeating the
arguments in Section~\ref{cinco} we obtain a positive solution
$u_\mu$ of (\ref{eq})
with $\tilde{I}_\mu(u_\mu)\leq\tilde{I}_\mu(\uu)$. The uniform
bound on the $L^\infty(\R^N)$ norm on $u_\mu$ follows from the
uniform coercivity in Lemma~\ref{coercivity} and~(\ref{unif}).
\end{proof}
We mention it is possible to construct examples where equation\/
{\rm (\ref{eq})} has a positive solution for a $\B$ growing faster
than in (\ref{limsup}) and an $\h$ without compact support.

\section{The case of a bounded domain}\label{bounded}

As we noted in the last paragraph of Section~\ref{cinco}, the
upper bound (\ref{limsup}) we imposed on $\B$ was the weakest one
under which our proof goes through. In this sense, the choice we
made for $\D$ in (\ref{def_d}) was the best one possible. To treat
the case of a bounded domain $\Omega$ we start by constructing the
best function $\D$ for this setting. This is done in the next
lemma. We note that in part {\bf (i)} we do not assume $\Omega$ is
bounded (having in mind future extensions to the case of unbounded
domains which are not the whole space $\R^N$). In fact, if one is
just concerned with the case of a bounded domain, then a shorter
proof of {\bf (i)} can be given.

\begin{Lem}\label{functiond}
Let $\Omega$ be a smooth domain in $\R^N$, $r>0$, $y_0\in\Omega$
with ${\rm dist}\,(y_0,\partial\Omega)>3r$, and $G$ be Green's
function of the first kind for $\Omega$. In \/~{\bf (ii)}
and\/~{\bf (iii)} assume $\Omega$ is bounded.

{\rm {\bf (i)}} There exists a function $\D\in
C^2(\overline{\Omega})$, superharmonic in $\Omega$ and harmonic in
$\Omega\setminus B_{r}(y_0)$, satisfying
\begin{equation}\label{dup}
cG(x,y_0)\leq\D(x)\leq CG(x,y_0)\quad{\rm for}\ x\in
\overline{\Omega}\setminus B_{2r}(y_0)
\end{equation}
for some constants $c,C>0$.\\
{\rm {\bf (ii)}} A function $\B :\Omega\to\R^+_0$ satisfies
\begin{equation}\label{ba}
\B\leq \overline{C}_1\A\, [{\rm
dist}\,(\,\cdot\,,\partial\Omega)]^{-\beta}.
\end{equation}
for some constant $\overline{C}_1>0$ if and only if the function
$\B$ satisfies
\begin{equation}\label{bainv}
\B\leq C_1\A{\D}^{-\beta}.
\end{equation}
for some constant\/ $C_1>0$ and the function $\D$ as in {\bf (i)}.\\
{\rm {\bf (iii)}} If $\tilde{\D}\in {\cal D}^{1,2}(\Omega)$ is
continuous, $\tilde{\D}\not\equiv 0$, $-\Delta\tilde{\D}\geq 0$
and $\B\leq\tilde{C}_1\A \tilde{\D}^{-\beta}$ for some constant
$\tilde{C}_1>0$, then $\B\leq C_1\A \D^{-\beta}$ for some constant
$C_1>0$.
\end{Lem}
\begin{proof} \ \\
{\bf (i)} Let
$$\Gamma(x)=\frac{1}{N(N-2)\omega_N}\cdot\frac{1}{|x|^{N-2}},$$
where $\omega_N$ is the volume of the unit ball in $\R^N$. The
function $\Gamma$ is uniformly continuous in $\R^N\setminus
B_r(0)$. This means for each $\eps>0$ there exists $0<\delta<r$
such that $y_1,y_2\in B_r(0)^C$ and $|y_1-y_2|<2\delta$ implies
$|\Gamma(y_1)-\Gamma(y_2)|<\eps$. If $y_1,y_2\in B_\delta(y_0)$
and $|x-y_1|\geq r$, $|x-y_2|\geq r$ then
$|\Gamma(x-y_1)-\Gamma(x-y_2)|<\eps$. Hence,
$$
y_1,y_2\in B_\delta(y_0)\ {\rm and}\
x\in\overline{\Omega}\setminus B_{r+\delta}(y_0)\ \Longrightarrow\
|\Gamma(x-y_1)-\Gamma(x-y_2)|<\eps.
$$
Green's function of the first kind for $\Omega$ is
$$
G(x,y)=\Gamma(x-y)+h_y(x),
$$
where
$$
\left\{\begin{array}{ll}
-\Delta h_y(x)=0&{\rm for}\ x\in\Omega,\\
h_y(x)=-\Gamma(x-y)&{\rm for}\ x\in\partial\Omega.
\end{array}\right.
$$
When $\Omega$ is unbounded, we further assume $h_y$ satisfies
$\lim_{x\to\infty}h_y(x)=0$. Then the existence of such a $h_y$
can be established by adapting Perron's method or applying
standard variational arguments. For $y_1,y_2\in B_\delta(y_0)$ and
$x\in\partial\Omega$, we have $|h_{y_1}(x)-h_{y_2}(x)|<\eps$, so
by the maximum principle
$$
y_1,y_2\in B_\delta(y_0)\ {\rm and}\
x\in\overline{\Omega}\setminus B_{r+\delta}(y_0)\ \Longrightarrow\
|G(x,y_1)-G(x,y_2)|<2\eps.
$$
One easily obtains $x\in\partial B_{r+\delta}(y_0)$ implies
$$
G(x,y_0)\geq\frac{1}{N(N-2)\omega_Nr^{N-2}}\left(\frac{1}{2^{N-2}}-\frac{1}{3^{N-2}}\right)
=:c>0.
$$
The value $c$ only depends on $r$ and $N$. Let
$$
C=\max_{x\in\partial B_{r+\delta}(y_0)}G(x,y_0).
$$
Choose $\eps=c/4$. We have,
$$
y\in B_\delta(y_0)\ {\rm and}\ x\in\partial B_{r+\delta}(y_0)\
\Longrightarrow\ \frac{c}{2}\leq G(x,y)\leq C+\frac{c}{2}.
$$
So $y\in B_\delta(y_0)$ and $x\in\partial B_{r+\delta}(y_0)$
implies
\begin{equation}\label{ys}
\frac{c}{2C}G(x,y_0)\leq G(x,y)\leq
\left(\frac{C}{c}+\frac{1}{2}\right)G(x,y_0).
\end{equation}
By the maximum principle the two inequalities of the last previous
line also hold for $x\in\overline{\Omega}\setminus
B_{r+\delta}(y_0)$. Let $\eta\in {\cal
D}\left(B_\delta(y_0)\right)$, $\eta\geq 0$ and $\int \eta=\rho>0$
and consider the function $\D\in {\cal D}(\overline{\Omega})$
defined by
\begin{equation}\label{defd}
\D(x)=\int G(x,y)\eta(y)\,dy.
\end{equation}
Multiplying (\ref{ys}) by $\eta(y)$ and integrating, for
$x\in\overline{\Omega}\setminus B_{r+\delta}(y_0)$,
$$
\rho\frac{c}{2C}G(x,y_0)\leq \D(x)\leq
\rho\left(\frac{C}{c}+\frac{1}{2}\right)G(x,y_0).
$$
Obviously $-\Delta \D=\eta$ in $\Omega$ and $\D=0$ on $\partial\Omega$.\\
{\bf (ii)} Let $(N_\sigma,{\rm proj})$ (with ${\rm
proj}:N_\sigma\to\partial\Omega$) be a tubular neighborhood of
$\partial\Omega$ in $\overline{\Omega}$ (see \cite[p.\ 35]{O})
with the length of the segment ${\rm proj}^{-1}(x)$ equal to
$\sigma$ for each $x\in\partial\Omega$. There exist $0<\sigma<{\rm
dist}\,(y_0,\partial\Omega)-2r$ and $c>0$ satisfying
\begin{equation}\label{Nd}
x\in N_\sigma\ \Longrightarrow\ -\,\frac{\partial
\D}{\partial\nu_{{\rm proj}\,x}}(x)\geq c.
\end{equation}
The vector $\nu_{{\rm proj}\,x}$ is the exterior outward unit
normal to $\partial\Omega$ at the point ${\rm proj}\,x$. Indeed,
suppose by contradiction there exist $\sigma_n\searrow 0$ and
$x_n\in N_{\sigma_n}$ satisfying
$$
-\,\frac{\partial \D}{\partial\nu_{{\rm
proj}\,x_n}}(x_n)\leq\frac{1}{n}.
$$
Modulo a subsequence, $x_n\to x_0\in\partial\Omega$. It follows
${\rm proj}\,x_n\to{\rm proj}\,x_0=x_0$, $\nu_{{\rm
proj}\,x_n}\to\nu_{{\rm proj}\,x_0}$ and $ -\,\frac{\partial
\D}{\partial\nu_{x_0}}(x_0)\leq 0. $ This contradicts Hopf's
Lemma. We have established (\ref{Nd}). Since $\D\in
C^2(\overline{\Omega})$, there exists $C>0$ such that
\begin{equation}\label{Ndt}
x\in N_\sigma\ \Longrightarrow\ -\,\frac{\partial
\D}{\partial\nu_{{\rm proj}\,x}}(x)\leq C.
\end{equation}
Given $x\in N_\sigma$, we integrate $\frac{\partial
\D}{\partial\nu_{{\rm proj}\,x}}$ along the part of the segment
${\rm proj}^{-1}({\rm proj}\,x)$ between ${\rm proj}\,x$ and $x$.
This part of ${\rm proj}^{-1}({\rm proj}\,x)$ has length ${\rm
dist}\,(x,\partial\Omega)$. Using (\ref{Nd}) and (\ref{Ndt}),
\begin{equation}\label{gx}
x\in N_\sigma\ \Longrightarrow\ c\,{\rm
dist}\,(x,\partial\Omega)\leq \D(x)\leq C\,{\rm
dist}\,(x,\partial\Omega).
\end{equation}
Suppose (\ref{ba}) holds. Using (\ref{gx}), $x\in N_\sigma\
\Rightarrow\ \B(x)\leq C\A(x)\, [\D(x)]^{-\beta}$. On the other
hand, there exist constants $c,C>0$ such that
$$
\overline{\Omega}\setminus N_\sigma\ \Longrightarrow\
c\,\frac{{\rm diameter(\Omega)}}{2}\leq \D(x) \leq C\sigma.
$$
As a consequence,
\begin{equation}\label{ddd}x\in
\overline{\Omega}\setminus N_\sigma\ \Longrightarrow\ c\,{\rm
dist}\,(x,\partial\Omega)\leq \D(x) \leq C{\rm
dist}\,(x,\partial\Omega).
\end{equation}
Taking into account (\ref{gx}) and (\ref{ddd}), we conclude
(\ref{ba}) and (\ref{bainv}) are equivalent.\\
{\bf (iii)} Suppose $\tilde{\D}\in {\cal D}^{1,2}(\Omega)$ is
continuous, $\tilde{\D}\not\equiv 0$ and $-\Delta\tilde{\D}\geq
0$. Multiplying the last inequality by $\tilde{\D}^-$ and
integrating, $\tilde{\D}^-\equiv 0$. From \cite[Theorem~8.19]{GT},
$\inf_{x\in B_{\delta}(y_0)}\tilde{\D}(x)>0$. Thus there exists
$C>0$ such that
\begin{equation}\label{ccc}
\tilde{\D}(x)\geq C\D(x)
\end{equation}
for $x\in\overline{B_{\delta}(y_0)}$. By the maximum principle, as
$\D$ is harmonic in $\Omega\setminus\overline{B_\delta(y_0)}$,
inequality (\ref{ccc}) also holds for
$x\in{\Omega}\setminus\overline{B_{\delta}(y_0)}$. So (\ref{ccc})
holds for $x\in{\Omega}$. The assertion follows.
\end{proof}

In the remainder of this section we suppose $\Omega$ is a smooth
bounded domain in $\R^N$, $N\geq 3$. We wish to prove the
existence of a positive solution to equation (\ref{eq}) where now
${\cal H}={\cal D}^{1,2}(\Omega)$. We introduce
\begin{enumerate}
\item[({\bf H}$\A$)$^{\prime\prime}$] The function
$\A:\Omega\to\R$ is positive and belongs to $L^\infty(\Omega)$.
\item[({\bf H}$\B$)$^{\prime\prime}$] The measurable function
$\B:\Omega\to\R$ is nonnegative, not identically equal to zero,
and satisfies
\begin{equation}\label{limsup2}
\bx \leq \overline{C}_1 \ax\, [{\rm dist}\,(\, \cdot\,
,\partial\Omega)]^{-\beta}.
\end{equation}
Let $B_0=\{x\in\Omega: \bxx =0\}$. We assume either $B_0$ has
measure zero, or $B_0=\overline{{\rm int}\,B_0}$ (closure in
$B_0$) with $\partial B_0$ Lipschitz. \item[({\bf
H}$\h$)$^{\prime\prime}$] The nonnegative and not identically
equal to zero function $\h$ belongs to the space $L^s(\R^N)$, for
some $s>N$.
\end{enumerate}
\begin{Rem}\label{po}
{\rm Proposition~\ref{exemplo}} generalizes to the case of a
bounded domain.
\end{Rem}
The proof is given in the Appendix.
\begin{Thm}\label{thm3}
Under\/ {\rm ({\bf H}$\A$)$^{\prime\prime}$}, {\rm ({\bf H}$g$)},
{\rm ({\bf H}$\B$)$^{\prime\prime}$}, {\rm ({\bf H}$\lambda$)}
and\/ {\rm ({\bf H}$\h$)$^{\prime\prime}$}, there exists\/
$\mu_4>0$ such that for all\/ $0\leq\mu\leq\mu_4$ equation\/ {\rm
(\ref{eq})} has a positive weak solution\/ $u_\mu\in {\cal H}\cap
C^{1,\alpha}(\Omega)$.
\end{Thm}
\begin{proof}
We fix any $x_1\in\Omega$ and $r_1<{\rm
dist}\,(x_1,\partial\Omega)/3$. Let $\D$ be as in {\bf (i)} of
Lemma~\ref{functiond} with $y_0=x_1$ and $r=r_1$. By {\bf (ii)} of
the same Lemma, the function $\B$ satisfies (\ref{bainv}). We
repeat the arguments in Section~\ref{especial} but with this new
function $\D$. For any nonnegative $\mu$ we obtain a solution
$\hu\in C^{1,\alpha}(\overline{\Omega})$ to (\ref{special}). As in
Lemma~\ref{negative} there exist $\mu_5, \Ceight>0$ such that for
$0\leq\mu\leq\mu_5$, we have $\inf_NI_\mu\leq-\Ceight<0$ (with $N$
as in (\ref{N})). As in the beginning of Section~\ref{psp}, there
exists $0<\mu_6\leq\mu_5$ such that for all $0\leq\mu\leq\mu_6$
one can choose $x_0(\mu)$ where $\hu(x_0(\mu))>0$. In addition,
there exists $\rho>0$ such that
$$
\inf_{\overline{B_\rho(x_0(\mu_6))}}\hus>0.
$$
Let $r_0<\min\{\rho,{\rm dist}\,(x_0(\mu_6),\partial\Omega)/3\}$.
We again use {\bf (i)} of Lemma~\ref{functiond}, but this time
with $y_0=x_0(\mu_6)$ and $r=r_0$, to construct a function
$\hat{\D}\in C^2(\overline{\Omega})$, superharmonic in $\Omega$
and harmonic in $\Omega\setminus B_{r_0}(x_0(\mu_6))$ satisfying
(\ref{dup}). We fix $\eps>0$ sufficiently small such that
$$
\eps\hat{\D}(x)\leq \hus(x)\quad{\rm for}\ x\in
B_\rho(x_0(\mu_6)).
$$
Clearly,
$$
\eps\hat{\D}(x)\leq (\hus+\mu_6w)(x)\quad{\rm for}\ x\in
B_\rho(x_0(\mu_6)).
$$
with $w$ as in (\ref{h}). The maximum principle implies
$$
\eps\hat{\D}(x)\leq(\hus+\mu_6 w)(x)\quad{\rm for}\
x\in\Omega\setminus B_\rho(x_0(\mu_6)).
$$
As in Section~\ref{psp}, we use $\hus$ as a subsolution to
(\ref{special}) when $0\leq\mu\leq\mu_6$. We minimize $I_\mu$ over
the set
$$
\left\{u\in{\cal H}:\hus\leq u\leq \llll\D \ {\rm \, a.e.\ in}\
\R^N\right\},
$$
where $\llll$ is as in (\ref{llll}), to obtain new solutions $\bu$
of (\ref{special}) for $0\leq\mu\leq\mu_6$ with $I_\mu(\bu)<0$.
These solutions satisfy
\begin{equation}\label{quase}
\eps\hat{\D}\leq\bu+\mu w.
\end{equation}
Combining (\ref{gx}) and (\ref{ddd}), there exist constants
$c,C>0$ such that
\begin{equation}\label{quasez}
c\,{\rm dist}\,(\,\cdot\,,\partial\Omega)\leq\hat{\D}\leq C\,{\rm
dist}\,(\,\cdot\,,\partial\Omega).
\end{equation}
On the other hand, since $\h\in L^s(\Omega)$ with $s>N$, $w\in
C^{1,\alpha}(\overline{\Omega})$.
Thus from (\ref{quase}) and (\ref{quasez}) there exists
$0<\mu_7\leq\mu_6$ such that for all $0\leq\mu\leq\mu_7$ the
function $\bu$ is positive in $\Omega$. Now we argue as in
Section~\ref{cinco} and use $\bu$ as subsolutions to (\ref{eq}).
For $0\leq\mu\leq\mu_7$ and all integers $m\geq 1$, we obtain a
positive solution $u^m_\mu$ of (\ref{mmm}) with
$I_\mu^m(u^m_\mu)\leq I_\mu(\bu)<0$. This time we use
\cite[Theorem~8.25]{GT} to conclude the $u^m_\mu$ are uniformly
bounded. Choosing any sufficiently large $m$ we obtain a positive
solution to (\ref{eq}).
\end{proof}

\section{Further extensions}\label{extensions}

The results of the previous sections may be generalized to prove
the existence of a positive solution to the equation
\begin{equation}\label{eq2}
-\Delta u=\lambda\A[u-g(\,\cdot\,,u)]-\mu\h,\qquad u\in {\cal H}.
\end{equation}
We give two results related to Theorems~\ref{thm} and \ref{thm2}
whose proofs we leave to the reader. First we replace ({\bf H}$g$)
and ({\bf H}$\B$) by
\begin{enumerate}
\item[({\bf H}$g$)$_\D$] The function $g:\R^N\times\R\to\R_0^+$ is
Carathéodory, with $g(x,s)=0$ for $x\in\R^N$ and $s\leq 0$. Let
$B_0=\left\{x\in\R^N:g(x,s)=0\ {\rm for}\ s\in\R\right\}$. We
assume either $B_0$ has measure zero, or $B_0=\overline{{\rm
int}\,B_0}$ with $\partial B_0$ Lipschitz. Furthermore, $ g\in
L^\infty_{{\rm loc}}(\R^N\times\R) $,
\begin{equation}\label{so}
\limsup_{s\to 0}\frac{[\D(x)]^\beta
g(x,s)}{s^{1+\beta}}<\infty\quad{\rm uniformly\ for}\ x\in\R^N,
\end{equation}
where $\beta>0$ is a fixed constant and $\D$ is defined in
(\ref{def_d}), and
$$
\lim_{s\to +\infty}\frac{g(x,s)}{s}=+\infty\quad{\rm for\ each}\
x\in B_0^C.
$$
\end{enumerate}
\begin{Thm}\label{thm4}
Under\/ {\rm ({\bf H}$\A$)}, {\rm ({\bf H}$g$)$_\D$}, {\rm ({\bf
H}$\lambda$)} and\/ {\rm ({\bf H}$\h$)}, there exists\/ $\mu_0>0$
such that for all\/ $0\leq\mu\leq\mu_0$ equation\/ {\rm
(\ref{eq2})} has a positive weak solution $u_\mu\in {\cal H}\cap
C^{1,\alpha}_{{\rm loc}}(\R^N)$. Furthermore, there exists\/
$\Cthree>0$ such that for all\/ $0\leq\mu\leq\mu_0$ this weak
solution\/ $u_\mu$ satisfies
$$
u_\mu(x)\geq\frac{\Cthree}{|x|^{N-2}}\quad{\rm for\ large}\ |x|.
$$
\end{Thm}

Now we replace ({\bf H}$g$), ({\bf H}$\B$) and ({\bf H}$\h$) as
follows:
\begin{enumerate}
\item[({\bf H}$g$)$_\varUpsilon$] The function
$g:\R^N\times\R\to\R_0^+$ is continuous, with $g(x,s)=0$ for
$x\in\R^N$ and $s\leq 0$. Let $B_0=\left\{x\in\R^N:g(x,s)=0\ {\rm
for}\ s\in\R\right\}$. We assume either $B_0$ has measure zero, or
$B_0=\overline{{\rm int}\,B_0}$ with $\partial B_0$ Lipschitz.
Furthermore, $g\in L^\infty_{{\rm loc}}(\R^N\times\R)$,
$$ 
\lim_{s\to 0}\frac{g(x,s)}{s}=0\quad{\rm uniformly\ for}\ x\ {\rm
in\ compact\ subsets\ of}\ \R^N,
$$ 
and
$$
\lim_{s\to +\infty}\frac{g(x,s)}{s}=+\infty\quad{\rm for\ each}\
x\in B_0^C.
$$
\item[({\bf H}$\h$)$^{\prime\prime\prime}$] The measurable,
nonnegative and not identically equal to zero function $\h$ has
compact support
and there exists a constant $C>0$ such that $\h\leq C\A$.
\end{enumerate}
\begin{Thm}\label{thm5}
Under\/ {\rm ({\bf H}$\A$)}, {\rm ({\bf H}$g$)$_\varUpsilon$},
{\rm ({\bf H}$\lambda$)} and\/ {\rm ({\bf
H}$\h$)$^{\prime\prime\prime}$}, there exists\/ $\mu_3>0$ such
that for all\/ $0\leq\mu\leq\mu_3$ equation\/ {\rm (\ref{eq2})}
has a positive weak solution \/ $u_\mu\in {\cal H}\cap
C^{1,\alpha}_{{\rm loc}}(\R^N)$.
\end{Thm}

\section{Appendix}\label{app}

\begin{altproof}{Proposition~\ref{exemplo}} \ \\
\noindent {\rm\bf (i)} We choose an $R>0$ such that
$B_R(0)\setminus B_0\neq\emptyset$. If the restriction  of $g$ to
$\R^+$ is positive, then $\B g(u)\chi_{B_R(0)}\not\equiv 0$. For
all $v\in{\cal D}(\R^N)$ with $v\geq 0$
\begin{equation}\label{lambdau}
\int\nabla u\cdot\nabla v\leq \lambda\int\A uv-\int\B
g(u)\chi_{B_R(0)}v-\mu\int \hx v.
\end{equation}
So (\ref{lambdau}) holds for all $v\in{\cal H}$ with $v\geq 0$.
Taking $v=u$ we obtain
$$\|u\|^2\leq\lambda\int \A u^2-\int\B g(u)u\chi_{B_R(0)}-\mu\int \hx u\leq\lambda\int \A u^2
$$
and the last inequality is strict if $\mu>0$ or if the restriction
of $g$ to $\R^+$ is positive. The conclusion follows.

\noindent {\rm\bf (ii)} Suppose $\h=0$ on $B_0$. We write
$u=u_0+u^\perp$ where $u_0|_{{\rm int}\,B_0}$ is the projection of
$u$ on ${\cal D}^{1,2}({\rm int}\,B_0)$ and $u_0=0$ on $({\rm
int}\,B_0)^C$. This means $u_0|_{{\rm int}\,B_0}\in {\cal
D}^{1,2}({\rm int}\,B_0)$ and
$$
\int\nabla u\cdot\nabla v=\int\nabla u_0\cdot\nabla v\qquad{\rm
for\ all}\ v\in{\cal D}^{1,2}({\rm int}\,B_0).
$$
The function $u^\perp:=u-u_0$ so that $u=u^\perp$ on $({\rm
int}\,B_0)^C$. Note
$$
\int\nabla u^\perp\cdot\nabla v=\int\nabla (u-u_0)\cdot\nabla
v=0\qquad{\rm for\ all}\ v\in{\cal D}^{1,2}({\rm int}\,B_0),
$$
which means that $u^\perp$ is harmonic in ${\rm int}\,B_0$. Since
$u$ is superharmonic in ${\rm int}\,B_0$ and $u^\perp$ is harmonic
in ${\rm int}\,B_0$, $u_0$ is superharmonic in ${\rm int}\,B_0$.
Thus $u_0$ is nonnegative. The function $u_0$ cannot be
identically zero. Otherwise in ${\rm int}\,B_0$ we would have
$0=-\Delta u^\perp=-\Delta u=\lambda\A u^\perp$. This implies
$u^\perp\equiv 0$ in ${\rm int}\, B_0$ and so $u\equiv 0$ in ${\rm
int}\, B_0$, contradicting the fact that $u$ is positive. The
function $u$ has a positive trace on $\partial B_0$. Also
$u=u^\perp$ on $\partial B_0$. So from $u^\perp\in{\cal H}$,
clearly $(u^\perp)^-|_{{\rm int}\,B_0}\in{\cal D}^{1,2}({\rm
int}\,B_0)$, and hence $(u^\perp)^-|_{{\rm int}\,B_0}\equiv 0$. By
the strong maximum principle $u^\perp>0$ on $B_0$. Let
\begin{equation}\label{phi}
\left\{\begin{array}{ll}
        -\Delta\phi_1^*=\lambda_*\A\phi_1^*&{\rm in}\ {\rm int}\,B_0,\\
        \phi_1^*>0&{\rm in}\ {\rm int}\,B_0,\\
    \phi_1^*=0&{\rm on}\ ({\rm int}\, B_0)^C.
      \end{array}
\right.
\end{equation}
One can easily see we may also take $v$ such that
$v|_{B_0}=\phi_1^*$ and $v|_{B_0^C}=0$ in (\ref{int}). Indeed,
this follows from $\B\in L^\infty_{{\rm loc}}(\R^N)$ and
$\phi_1^*|_{{\rm int}\,B_0}\in {\cal D}^{1,2}({{\rm int}}\,B_0)$.
We obtain
$$
\int\nabla u_0\cdot\nabla\phi_1^*+\int\nabla
u^\perp\cdot\nabla\phi_1^*= \lambda\int\A
u_0\phi_1^*+\lambda\int\A u^\perp\phi_1^*.
$$
This yields
$$
\lambda_*\int\A u_0\phi_1^*= \lambda\int\A
u_0\phi_1^*+\lambda\int\A u^\perp\phi_1^*>\lambda\int\A
u_0\phi_1^*,
$$
and so $\lambda<\lambda_*$.

\noindent {\rm\bf (iii)} We give functions $\A$, $\B$, $g$, $\h$
(with $\h\not\equiv 0$ on $B_0$), and a function $u\in{\cal H}$
which is a positive solution of (\ref{eq}) for
$\lambda=\lambda_*+\mu$. Here $\mu> 0$ is the parameter in
(\ref{eq}). Since all functions will be radially symmetric, we
introduce the coordinate $r=|x|$ and write them in terms of $r$.
We choose the set $B_0=\left\{x\in\R^N:r\leq 1\right\}$. The
functions $\A$ and $g$ are
$$
\A(r) =\left\{\begin{array}{ll}
        1&{\rm for}\ r\leq 1,\\
    \frac{1}{r^{(N-2)\beta}}&{\rm for}\ r>1,
      \end{array}
\right.
$$
$$
g(u) =\left\{\begin{array}{ll}
        0&{\rm for}\ u\leq 0,\\
    u^{1+\beta}&{\rm for}\ u>0,
      \end{array}
\right.
$$
with $\beta>2$. We define $u$ using (\ref{phi}),
$$
u(r) =\left\{\begin{array}{ll}
        \phi_1^*+\kappa&{\rm for}\ r\leq 1,\\
    \frac{\kappa}{r^{N-2}}&{\rm for}\ r>1,
      \end{array}
\right.
$$
with $\kappa=\left.-\frac{1}{N-2}\frac{\partial\phi_1^*}{\partial
r}\right|_{r=1}$ so that $u\in C^1(\R^N)$. This is possible
because $\phi_1^*$ is spherically symmetric (\cite{GNN}) and
$\left.\frac{\partial\phi_1^*}{\partial r}\right|_{r=1}<0$ (by
Hopf's lemma). The functions $\B$ and $\h$ are
$$
\B(r) =\left\{\begin{array}{ll}
        0&{\rm for}\ r\leq 1,\\
    \frac{\lambda}{\kappa^\beta}&{\rm for}\ r>1,
      \end{array}
\right.
$$
$$
\mu\h(r) =\left\{\begin{array}{ll}
        \mu\phi_1^*(r)+\lambda\kappa&{\rm for}\ r\leq 1,\\
    0&{\rm for}\ r>1.
      \end{array}
\right.
$$
Our assumptions are all satisfied except for ({\bf H}$\lambda$) of
course. In particular, the function $\A$ is positive and belongs
to $L^{N/2}(\R^N)\cap L^\infty(\R^N)$. The measurable function
$\B$ is nonnegative, not identically equal to zero, and satisfies
(\ref{limsup}) for $C_1=\frac{\lambda}{\kappa^\beta}$ as $\ax
{\dx}^{-\beta}>1$. Note also $u\in{\cal H}$. The function $u$
satisfies (\ref{eq}) in $B_1(0)$ and in $\overline{B_1(0)}^C$. In
fact, for $r<1$,
$$
-\Delta(\phi_1^*+\kappa)=\lambda\cdot
1\cdot(\phi_1^*+\kappa)-0-(\mu\phi_1^*+\lambda\kappa)
=\lambda_*\phi_1^*.
$$
For $r>1$,
$$
0=\lambda\frac{1}{r^{(N-2)\beta}}\frac{\kappa}{r^{N-2}}
-\frac{\lambda}{\kappa^\beta}\frac{\kappa^{1+\beta}}{r^{(N-2)(1+\beta)}}-0.
$$
Let $v\in{\cal D}(\R^N)$. We recall $u\in C^1(\R^N)$. Multiplying
(\ref{eq}) by $v$ and integrating over $B_1(0)$ we obtain
\begin{equation}\label{intu}
-\int_{\partial B_1(0)}\frac{\partial u}{\partial r}v
+\int_{B_1(0)}\nabla u\cdot\nabla v=\lambda\int_{B_1(0)} \ax
uv-\int_{B_1(0)} \bx g(u)v-\mu\int_{B_1(0)} \hx v.
\end{equation}
Multiplying (\ref{eq}) by $v$ and integrating over
$\overline{B_1(0)}^C$ we obtain
\begin{equation}\label{intd}
\int_{\partial B_1(0)}\frac{\partial u}{\partial r}v
+\int_{\overline{B_1(0)}^C}\nabla u\cdot\nabla
v=\lambda\int_{\overline{B_1(0)}^C} \ax uv
-\int_{\overline{B_1(0)}^C} \bx
g(u)v-\mu\int_{\overline{B_1(0)}^C} \hx v.
\end{equation}
Adding (\ref{intu}) and (\ref{intd}), the function $u$ is a
positive weak solution of (\ref{eq}).
\end{altproof}

\begin{altproof}{Remark~\ref{po}}
The proof of items {\bf (i)} and {\bf (ii)} is similar to the case
of the space $\R^N$. To check item {\rm\bf (iii)} let
$\Omega=B_2(0)$. We may take
$$
\A(r) =\left\{\begin{array}{ll}
        1&{\rm for}\ r\leq 1,\\
    \left(\frac{1}{r^{N-2}}-\frac{1}{2^{N-2}}\right)^{\beta}&{\rm for}\ 1<r<2,
      \end{array}
\right.
$$
$$
u(r) =\left\{\begin{array}{ll}
        \phi_1^*+\kappa\left(1-\frac{1}{2^{N-2}}\right)&{\rm for}\ r\leq 1,\\
    \kappa\left(\frac{1}{r^{N-2}}-\frac{1}{2^{N-2}}\right)&{\rm for}\ 1<r<2,
      \end{array}
\right.
$$
$$
\mu\h(r) =\left\{\begin{array}{ll}
        \mu\phi_1^*(r)+\lambda\kappa\left(1-\frac{1}{2^{N-2}}\right)&{\rm for}\ r\leq 1,\\
    0&{\rm for}\ 1<r<2,
      \end{array}
\right.
$$
and all the parameters and other functions as in the proof of
Proposition~\ref{exemplo}. There exists $\overline{C}_1>0$ such
that (\ref{ba}) holds because
$$
0<\lim_{r\to
2}\left[\left(\frac{1}{r^{N-2}}-\frac{1}{2^{N-2}}\right)\frac{1}{2-r}\right]^\beta<\infty.
$$
\end{altproof}

\end{document}